\documentclass[a4paper,reqno]{amsart}

\usepackage[english]{babel}
\usepackage{amsthm, latexsym, gastex, eufrak, amssymb,url,listings,enumitem}
\usepackage[mathcal]{eucal}
\usepackage{graphicx,url}

\copyrightinfo{2017}{Simone Ugolini}

\DeclareMathOperator{\lcm}{lcm}

\DeclareMathOperator{\End}{End}

\DeclareMathOperator{\Pre}{Pre}

\renewcommand{\phi}[0]{\varphi}
\renewcommand{\theta}[0]{\vartheta}
\renewcommand{\epsilon}[0]{\varepsilon}

\newcommand{\B}{\text{$\mathfrak{B}$}}
\newcommand{\N}{\text{$\mathbf{N}$}}
\newcommand{\Z}{\text{$\mathbf{Z}$}}
\newcommand{\Q}{\text{$\mathbf{Q}$}}

\newcommand{\Pro}{\text{$\mathbb{P}^1$}}

\newcommand{\F}{\text{$\mathbb{F}$}}

\newtheorem{theorem}{Theorem}[section]
\newtheorem{lemma}[theorem]{Lemma}

\theoremstyle{definition}

\newtheorem{example}[theorem]{Example}

\theoremstyle{remark}
\newtheorem{remark}[theorem]{Remark}

\numberwithin{equation}{section}

\begin{document}

\bibliographystyle{amsplain}

\date{}
\title[]
{Functional graphs of rational maps induced by endomorphisms of ordinary  elliptic curves over finite fields}

\author{S.~Ugolini}
\email{sugolini@gmail.com}

\begin{abstract}
In this paper we study the dynamics of rational maps induced by endomorphisms of ordinary elliptic curves defined over finite fields.  
\end{abstract}

\maketitle
\section{Introduction}
The study of the dynamics of polynomial and rational maps over finite fields has attracted the attention of many investigators. Motivated also by Pollard's rho algorithm, Rogers \cite{rog} and  Vasiga and Shallit \cite{vas} dealt with certain quadratic maps. Later other studies on monomial, polynomial and rational maps  appeared (see for example \cite{chsh, coetal, gas, paqu, shasu}).

In two different articles \cite{SU2, SU35} we studied the dynamics of the maps $x \mapsto x+x^{-1}$ over finite fields of characteristic $2$, $3$ and $5$. Later \cite{SUk} we extended such results to certain maps of the form $x \mapsto k \cdot (x+x^{-1})$. In all these cases the maps studied  are related to certain endomorphisms of  ordinary elliptic curves defined over finite fields and having complex multiplication.  

The idea behind the current work is to generalize the studies of our three aforementioned papers in an unified frame. More precisely, we aim at studying the functional graphs of rational maps induced by endomorphisms of ordinary elliptic curves defined over finite fields. 

Consider an ordinary elliptic curve $E$ defined over the field $\F_q$ with $q$ elements, where $q$ is a power of a prime $p$, and an endomorphism $\alpha$ of $E$ defined as
\begin{equation}\label{alpha}
\alpha(x,y) := (\alpha_1(x), y \cdot \alpha_2 (x))
\end{equation}
for some rational functions $\alpha_1$ and $\alpha_2$ in $\F_q (x)$, where $\alpha_1 (x) = \frac{a(x)}{b(x)}$ for some polynomials $a(x)$ and $b(x)$ in $\F_q [x]$.

Let $n$ be a positive integer and consider the map $r$ defined over  $\Pro(\F_{q^n}) := \F_{q^n} \cup \{ \infty \}$ as 
\begin{displaymath}
r(x) :=
\begin{cases}
\infty & \text{if $x= \infty$ or $b(x) = 0$;}\\
\alpha_1(x) & \text{otherwise.}
\end{cases}
\end{displaymath}

We want to study the properties of the digraph $G^{q^n}$ whose vertices are the elements of $\Pro(\F_{q^n})$ and where an arrow connects two elements $x_1$ and $x_2$ provided that $x_2 = r(x_1)$. 

A detailed description of the structure of $G^{q^n}$ is given in Section \ref{str}. Very briefly, $G^{q^n}$ is formed by a finite number of connected components. Each component consists of a cycle, whose vertices are roots of reversed trees. We say that an element $\tilde{x} \in \Pro(\F_{q^n})$ is $r$-periodic if $r^m (\tilde{x}) = \tilde{x}$ for some positive integer $m$ (here $r^m$ denotes the $m$-fold composition of $r$ with itself).

\subsection{General settings of the paper}
In the paper $E$ is an ordinary elliptic curve defined by an equation of the form
\begin{equation}\label{equa_ell}	
y^2 + a_1 x y + a_3 y = x^3 + a_2 x^2 + a_4 x + a_6
\end{equation}
for some coefficients $a_i \in \F_q$, where $q$ is a power of a prime.

The endomorphism ring $\End(E)$ of $E$ over $\F_q$ is isomorphic to an order in an imaginary quadratic number field $\Q(\sqrt{d})$ for some square-free negative integer $d$ (see \cite{len}). In particular, the unique maximal order in $K:=\Q(\sqrt{d})$ is the ring of integers $\mathcal{O}_K$ of $K$. In this the paper we assume that $\End(E) = \mathcal{O}_K$.

\subsection{Structure of the paper}
The paper is organized as follows.
\begin{itemize}
\item In Section \ref{sec_back} we recall some general facts on the ring $R$ of the endomorphisms of $E$ over $\F_q$ and its quotients with respect to ideals.
\item Section \ref{str} contains the main results of the paper. In Section \ref{cyc_sec} we describe the cycles of the graph $G^{q^n}$, while in Section \ref{tree_sec} we describe the structure of the trees.
\item In Section \ref{exm} we analyse the structure of $G^{q^n}$ choosing three different elliptic curves and endomorphisms.
\end{itemize}

\subsection{Notation} For the sake of brevity, if $i$ and $j$ are two non-negative integers such that $i < j$ we define 
\begin{align*}
[i,j] & := \{x \in \N: i \leq x \leq j \}.
\end{align*}

\section{Background}\label{sec_back}
Let $R$ be the endomorphism ring of $E$ over $\F_q$. According to the assumptions made in the introduction $R = \Z [\omega_d]$, where  

\begin{displaymath}
\omega_{d} := 
\begin{cases}
\frac{1+\sqrt{d}}{2} & \text{if $d \equiv 1 \pmod{4}$,}\\
\sqrt{d} & \text{otherwise.}
\end{cases}
\end{displaymath}

For any ideal $I$ in $R$ we denote by $N(I)$ the absolute norm of $I$, namely 
\begin{equation*}
N(I) := |R/I|.
\end{equation*}

Let $m := |E(\F_{q})|$ and $d':= (q+1-m)^2-4q$. 

As proved by Wittmann \cite{wit}, $d' < 0$ and the $q$-Frobenius endomorphism $\pi_q$ which takes an element $(x,y) \in \overline{\F}_q^2$ to $(x^q, y^q)$ is represented in $R$ by 

\begin{equation*}
\pi_q = \frac{q+1-m+ f_0 \sqrt{d}}{2}
\end{equation*} 

for some positive integer $f_0$ such that $f_0^2 d = d'$.

\subsection{Dedekind domains and ideal factorization}\label{dede_sec}
Since $R$ is a Dedekind domain, any ideal $I$ in $R$ factors into prime ideals and $R/I$ is isomorphic to
\begin{equation}\label{dede_fact}
\left( \prod_{i=1}^{j} R / \B_i^{e_i} \right) \times \left( \prod_{i=j+1}^{k} R / \B_i^{e_i} \right) \times \left(\prod_{i=k+1}^{l} R / \B_i^{e_i}  \right), 
\end{equation}
where
\begin{itemize}
\item $j, k$ and $l$ are non-negative integers;
\item any $e_i$ is a positive integer for $i \in [1,l]$;
\item $\B_i$ is a prime ideal above a splitting rational prime $p_i$ for any $i \in [1,j]$;
\item $\B_i$ is a prime ideal above an inert rational prime $p_i$ for any $i \in [j+1,k]$;
\item $\B_{i}$ is a prime ideal above a ramified rational prime $p_{i}$ for any $i \in [k+1,l]$.
\end{itemize}
Moreover we adopt the following conventions for the integers $j, k$ and $l$:
\begin{itemize}
\item if no prime ideal above a splitting (resp. inert, ramified) rational prime appears in the factorization of $I$, then $j := 0$ (resp. $k := j$, $l:=k$);
\item if $k \not = j$ (resp. $l \not = k$), then $k > j$ (resp. $l > k$ and $l > j$).
\end{itemize}

\subsection{Galois rings}
The main source for the definitions and the results presented in this section is \cite{wan}.
  
A Galois ring is defined to be a finite ring with identity $1$ such that the set of its zero divisors with $0$ added forms a principal ideal $(p1)$ for some prime number $p$ (\cite[Definition 14.1]{wan}).

In the following $p$ is a rational prime, while $s$ and $m$ are two positive integers. Moreover, for any positive integer $a$, we denote the ring $\Z / p^a \Z$ by $\Z_{p^a}$.

Let $h$ be a monic polynomial in $\Z_{p^s} [x]$. We say that $h$ is basic irreducible (resp. basic primitive) in $\Z_{p^s} [x]$ if the polynomial $\overline{h}$ obtained reducing the coefficients of $h$ modulo $p$ is irreducible (resp. primitive) in $\Z_p [x]$. 

If $h$ is a monic basic irreducible polynomial of positive degree $m$ over $\Z_{p^s}$, then $\Z_{p^s} [x] / (h(x))$ is a Galois ring of characteristic $p^s$ and cardinality $p^{sm}$ (\cite[Example 14.2]{wan}), which we denote by $GR(p^s, p^{sm})$. Indeed, as explained in \cite[Section 2]{wan}, any Galois ring has characteristic a power of a prime and the following holds.

\begin{theorem}
Let $G$ be a Galois ring of characteristic $p^s$ and cardinality $p^{sm}$. Then $G$ is isomorphic to the ring $\Z_{p^s} [x] / (h(x))$ for any monic basic irreducible polynomial $h(x)$ of degree $m$ over $\Z_{p^s}$.
\end{theorem}

The following theorem is based on \cite[Theorem 14.8]{wan} and \cite[Corollary 14.9]{wan}.

\begin{theorem}\label{galo_theo}
\mbox{}

\begin{enumerate}
\item In the Galois ring $GR(p^s, p^{sm})$ there exists a nonzero element $\xi$of order $p^m-1$, which is a root of a monic basic primitive polynomial $h(x)$ of degree $m$ over $\Z_{p^s}$ and dividing $x^{p^m-1} - 1$ in $\Z_{p^s} [x]$.
\item Any element $c \in GR(p^s, p^{sm})$ can be written uniquely as 
\begin{equation*}
c = c_0 + c_1 p + \dots + c_{s-1} p^{s-1},
\end{equation*}
where $c_0, c_1, \dots, c_{s-1} \in \Gamma := \{0, 1, \xi, \xi^2, \dots, \xi^{p^m-2} \}$.
\item All the elements $c$ with $c_0=0$ are either zero divisors or zero, and they form the ideal $(p)$ of $GR(p^s, p^{sm})$.
\item Every nonzero element $c \in GR(p^s, p^{sm})$ can be expressed uniquely in the form $c = u p^r$, where $u$ is a unit and $0 \leq r < s$.
\end{enumerate}
\end{theorem}

\begin{remark} We notice in passing that $\Z_{p^s}$ is a Galois ring on its own. Indeed $\Z_{p^s} \cong GR (p^s, p^s)$.
\end{remark} 

\subsection{Structure of the quotient rings $R / \B_i^{e_i}$}
The main references for this section are \cite{eir, eir2}. Still, the notation for Galois rings used in \cite{eir, eir2} is slightly different from the notation adopted in this paper, which is the same as in \cite{wan}.

We distinguish three cases.

\textbf{Case 1: $i \in [1,j]$.}
According to \cite[Theorem 5]{eir}, there is an isomorphism
\begin{equation*}
R / \B_i^{e_i} \cong \Z_{p_i^{e_i}} = GR (p_i^{e_i}, p_i^{e_i}).
\end{equation*}

\textbf{Case 2: $i \in [j+1,k]$.}
According to \cite[Theorem 6]{eir} there is an isomorphism
\begin{equation*}
R / \B_i^{e_i} \cong GR (p_i^{e_i}, p_i^{2 e_i}).
\end{equation*}

\textbf{Case 3: $i \in [k+1,l]$.} Let $x+c$ be a factor in $\Q (\sqrt{d})[x]$ of the minimal polynomial of $d$ over $\Q$. The following hold in accordance with \cite[Theorem 1]{eir2}. 
\begin{itemize}
\item If $e_i = 1$, then the additive group of $R / \B_{i}^{e_{i}}$ is isomorphic to a cyclic group of order $p_i$.
\item If $e_{i}$ is even, then the additive group of $R / \B_{i}^{e_{i}}$ is isomorphic to the direct sum of two cyclic groups and its elements can be represented as
\begin{equation*}
z_0 + z_1 (d+c),
\end{equation*}
where $(d+c)$ is a principal ideal in  $\Z [\omega_d]$ and $z_0, z_1 \in \Z_{p_i^{e_{i}/2}}$.  
\item If $e_{i}$ is odd, then the additive group of $R / \B_{i}^{e_{i}}$ is isomorphic to the direct sum of two cyclic groups and its elements can be represented as
\begin{equation*}
x_0 + z_0 (d+c),
\end{equation*}
where $x_0 \in  \Z_{p_i^{(e_{i}+1)/2}}$ and $z_0 \in \Z_{p_i^{(e_{i}-1)/2}}$.  
\end{itemize}

\section{Structure of the graphs}\label{str}
According to \cite[Theorem 1]{len} there is an isomorphism   
\begin{equation*}
E(\F_{q^n}) \cong R / (\pi_q^n-1) R
\end{equation*}
of $R$-modules.

The following trivially holds because the equation (\ref{equa_ell}) has degree two.
\begin{lemma}
Let $\tilde{x} \in \F_{q^n}$. Then there exist two different or coincident points in $E(\F_{q^{2n}})$ having such an $x$-coordinate. 
\end{lemma}

Consider the sets
\begin{align*}
E_0 & := \{x \in \F_{q^n}: (x,0) \in E(\F_{q^n}) \} \cup \{ \infty \},\\
E_0 (\F_{q^n}) & := \{(x,0) : x \in \F_{q^n} \} \cup \{ O \},
\end{align*}
where $O$ is the point at infinity of $E$.

In analogy with \cite{SUk} we introduce the sets

\begin{eqnarray*}
A_n & := & \{x \in \F_{q^n} : (x,y) \in E(\F_{q^n}) \text{ for some $y \in \F_{q^n}$} \} \cup \{\infty \},\\
B_n & := & \{x \in \F_{q^n} : (x,y) \in E(\F_{q^{2n}}) \text{ for some $y \in \F_{q^{2n}} \backslash \F_{q^n}$} \} \cup E_0,
\end{eqnarray*}
and correspondingly the set

\begin{eqnarray*}
E(\F_{q^{2n}})_{B_n} & := & \{(x,y) \in E(\F_{q^{2n}}) : x \in B_n \}.
\end{eqnarray*}

We notice that 
\begin{equation*}
A_n \cap B_n = E_0
\end{equation*}
and the following holds in analogy with \cite[Lemma 3.3]{SUk}. 

\begin{lemma}\label{lem_iff}
Let $P$ be a rational point of $E(\F_{q^{2n}})$ such that either $P = O$ or $P = (\tilde{x}, \tilde{y})$ for some $\tilde{x} \in \F_{q^n}$ and $\tilde{y} \in \F_{q^{2n}}$. Then
\begin{eqnarray*}
(\pi_q^n-1) P = O & \Leftrightarrow  & P \in E(\F_{q^n});\\
(\pi_q^n+1) P = O & \Leftrightarrow  & P \in E(\F_{q^{2n}})_{B_n}.
\end{eqnarray*}
\end{lemma}

In virtue of Lemma \ref{lem_iff} we can say that there is an isomorphism
\begin{equation*}
E(\F_{q^{2n}})_{B_n} \cong R/(\pi_q^n+1) R.
\end{equation*}

Let $\alpha$ and $r$ be defined as in the Introduction. Along the way we denote by $\alpha$ also its representation as an element of $R$.

The following holds in analogy with \cite[Lemma 4.4]{SUk}.

\begin{lemma}\label{lem_a_n_b_n}
If $\tilde{x} \in A_n$ (resp. $\tilde{x} \in B_n$), then $r(\tilde{x}) \in A_n$ (resp. $r(\tilde{x}) \in B_n$).
\end{lemma}

\begin{remark}
While we can conveniently study the action of the map $r$ separately on the elements of $A_n$ and $B_n$, we should  take care of the fact that $A_n \cap B_n = E_0$.
\end{remark}

The following holds for the action of $r$ on $E_0$.

\begin{lemma}\label{lem_e_0}
If $\tilde{x} \in E_0$, then $r(\tilde{x}) \in E_0$.
\end{lemma}
\begin{proof}
The result follows from the fact that 
\begin{equation*}
\alpha(\tilde{x},0) = (\alpha_1(\tilde{x}), 0)
\end{equation*}
and $r(\infty) = \infty$.
\end{proof}

We set $S:=A_n$ or $S:=B_n$ and correspondingly $\delta:=\pi_q^n-1$ or $\delta:=\pi_q^n+1$.
  
The ring  $R / (\delta)$ is isomorphic to
\begin{equation*}
R / I_c \times R / I_t, 
\end{equation*}
where $I_c$ and $I_t$ are two ideals such that $(\alpha) \not\subseteq \B_c$ for any prime factor $\B_c$ of $I_c$, while $(\alpha) \subseteq \B_t$ for any prime factor $\B_t$ of $I_t$.

\subsection{Cycles of $S$} \label{cyc_sec}
If $\tilde{x}$ is $r$-periodic in $S$, then $\tilde{x}$ is the $x$-coordinate of a point 
\begin{equation*}
(P, [0]) \in R / I_c \times R / I_t.
\end{equation*} 
Suppose that 
\begin{equation}\label{I_c_fac}
R / I_c \cong \left( \prod_{i=1}^{j} R / \B_i^{e_i} \right) \times \left( \prod_{i=j+1}^{k} R / \B_i^{e_i} \right) \times \left(\prod_{i=k+1}^{l} R / \B_i^{e_i}  \right),
\end{equation}
where the notations are as in Section \ref{dede_sec}.
The description of the cycles formed by the elements of $S$ is similar to \cite[Section 3.1]{SU35}. 

For any $i \in [1,l]$ we define the  interval $H_i$ as follows: 
\begin{itemize}
\item $H_i := [0,e_i]$ if $i \in [1,k]$;
\item $H_i := [0, \lfloor (e_i+1)/2 \rfloor]$ if $i \in [k+1,l]$.
\end{itemize}

Moreover we define 
\begin{equation*}
H := \prod_{i=1}^{l} H_{i}. 
\end{equation*}

Let $h_i$ be an integer in $H_i$ for some $i \in [1,l]$.

We introduce the notation $n_{h_i}$  distinguishing three cases.
\begin{enumerate}
\item If $i \in [1, j]$, then
\begin{displaymath}
n_{h_i} :=
\begin{cases}
1 & \text{if $h_i=0$,}\\
p_i^{h_i} - p_i^{h_i-1} & \text{otherwise.}
\end{cases}
\end{displaymath}
\item If $i \in [j+1, k]$, or $i \in [k+1,l]$ and $e_{i}$ is even, then
\begin{displaymath}
n_{h_i} :=
\begin{cases}
1 & \text{if $h_i=0$,}\\
{p_i}^{2 h_i} - {p_i}^{2(h_{i}-1)} & \text{otherwise.}
\end{cases}
\end{displaymath}
\item If $i \in [k+1,l]$ and $e_{i}$ is odd, then
\begin{displaymath}
n_{h_{i}} :=
\begin{cases}
1 & \text{if $h_{i}=0$,}\\
{p_{i}}^{2 h_{i}} - {p_{i}}^{2(h_{i}-1)} & \text{if $1 \leq h_{i} \leq (e_{i}-1)/2$,}\\
{p_{i}}^{2 h_{i} - 1} - {p_{i}}^{2(h_{i}-1)} & \text{if $h_{i} = (e_{i}+1)/2$.}
\end{cases}
\end{displaymath} 
\end{enumerate}

Let  
\begin{displaymath}
s_{h_i} := \min \{v \in \N \backslash \{ 0 \} : \alpha^v \pm 1 \in \B_i^{h_i} \},
\end{displaymath}
where we adopt the convention that $\B_i^0 = R$. 

A $r$-periodic element $\tilde{x} \in S \backslash \{\infty \}$ is the $x$-coordinate of a rational point whose image in $R / I_c$ is $P = (P_i)_{i=1}^{l}$.

Any $P_i$ has additive order $p_i^{h_i}$ in $R / \B_i^{e_i}$ for some $h_i \in H_i$. The $l$ integers $h_i$, for $i \in [1,l]$, form the $l$-tuple
\begin{equation*}
h := (h_1, \dots, h_l).
\end{equation*}

The following holds.

\begin{lemma}\label{tech_1}
\mbox{}

\begin{enumerate}
\item Let $i \in [1,l]$. Then in $R / \B_i^{e_i}$ there are $n_{h_i}$ points of additive order $p_i^{h_i}$ for any $h_i \in H_i$.
\item Let $P_i$ be a point of order $p_i^{h_i}$ in $R / \B_i^{e_i}$. Then the smallest of the positive integers $v$ such that $[\alpha]^{v} P_i \in \{ \pm P_i \}$ is
$s_{h_i}$. 
\end{enumerate}
\end{lemma}
\begin{proof}
We prove separately the claims.
\begin{enumerate}
\item Let $i \in [1,k]$. A point $P_i$ has additive order $p_i^{h_i}$ greater than $1$ in $R / \B_i^{e_i}$ if and only if its representative in $GR(p_i^{e_i}, p_i^{m e_i})$, where $m \in \{1, 2 \}$, is 
\begin{equation}\label{P_i}
c_{e_i - h_i} p_i^{e_i - h_i} + \dots + c_{e_i - 1} p_i^{e_i-1}
\end{equation}
for some coefficients $c_{e_i - h_i}, \dots, c_{e_i -1} \in \{0, 1, \dots, \xi^{p_i^m-2} \}$ with $c_{e_i - h_i} \not = 0$. Therefore there are 
\begin{equation*}
(p_i^m-1) \cdot p_i^{m (h_i-1)} = p_i^{m h_i} - p_i^{m (h_i-1)}
\end{equation*}
points of additive order $p_i^{h_i}$ in $R/ \B_i^{e_i}$.

If $i \in [k+1,l]$, then the result follows from the fact that $\Z_{p_i^s} \cong GR(p_i^s, p_i^s)$ for any positive integer $s$.

\item Suppose that $i \in [1,k]$. Let $v$ be a positive integer such that $[\alpha]^v P_i = \pm P_i$, namely $[\alpha^v \pm 1] P_i = [0]$ in $R / \B_i^{e_i}$. 

In $GR(p_i^{e_i}, p_i^{m e_i})$, where $m \in \{1, 2 \}$, we can represent $P_i$ in the form (\ref{P_i}). Therefore, either $[\alpha^v \pm 1] = [0]$ or, according to Theorem \ref{galo_theo}(4), $[\alpha^v \pm 1]$ can be written in $GR(p_i^{e_i}, p_i^{m e_i})$ as   
\begin{equation*}
u p_i^{h_i}
\end{equation*}
for some unit $u$. Hence $[\alpha^v \pm 1] = [0]$ in $R / \B^{h_i}$, namely $\alpha^v \pm 1 \in \B_i^{h_i}$.

If $i \in [k+1,l]$, then the result follows in a similar way since $\Z_{p_i^s} \cong GR(p_i^s, p_i^s)$ for any positive integer $s$.
\end{enumerate}
\end{proof}

We define
\begin{displaymath}
s'_h := \lcm (s_{h_1}, \dots, s_{h_{l}}).
\end{displaymath}

As a consequence of Lemma \ref{tech_1} the period of $\tilde{x}$ with respect to the action of $r$ is $s_h$, where 
\begin{displaymath}
s_h:=
\begin{cases}
s'_h & \text{if $[\alpha]^{s'_h} P \in \{ \pm P \}$;}\\
2 s'_h & \text{otherwise.} 
\end{cases}
\end{displaymath}

Consider an element $h = (h_i)_{i=1}^{l}$ in $H$. Let $s_h$ be defined as above. 

We suppose that none of the following conditions is satisfied:
\begin{itemize}
\item $\prod_{i=1}^{l} p_i^{h_i} \in \{1, 2 \}$;
\item $\prod_{i=1}^{l} p_i^{h_i} = 4$ and $h_i \leq 1$ for any $i \in [1,l]$.
\end{itemize}

Indeed, we notice that the only point of order $1$ is the point at infinity, while the rational points having order $2$ have the $y$-coordinate equal to $0$ and consequently their $x$-coordinate belongs to $E_0$.

The number of points $(P, [0])$ in $R / (\delta)$ whose corresponding $x$-coordinate belongs to a cycle of length $s_h$ is $\prod_{i=1}^{l} n_{h_i}$.

If we denote by $C^S_h$ the set of all cycles in $G^{q^n}$ formed by such points, then the following holds in analogy with \cite[Theorem 3.3]{SU35}.

\begin{theorem}\label{thm_c_h}
If $h$ and $C^S_h$ are as above, then
\begin{equation*}
\lvert C^S_{h} \rvert = \frac{1}{2 s_h} \cdot \prod_{i=1}^{l} n_{h_i}.
\end{equation*}
\end{theorem} 

\subsection{Trees rooted in elements of S}\label{tree_sec}
If $x' \in \Pro(\F_{q^n})$ is not $r$-periodic, then it is preperiodic, namely some iterate of $x'$ is $r$-periodic. Therefore, if 
\begin{equation*}
Q = (Q_c, Q_t) \in R / I_c \times R / I_t
\end{equation*}
is a point in $E(\F_{q^{2n}})$ whose $x$-coordinate is $x'$, then 
\begin{equation*}
[\alpha]^s Q = ([\alpha]^s Q_c, [0])
\end{equation*}
for some positive integer $s$.

Hence the trees rooted in elements of $S$ can be studied relying upon the structure of $R/ I_t$, which we suppose to be isomorphic to
\begin{equation}\label{I_t_fac}
\left( \prod_{i=1}^{j} R / \B_i^{e_i} \right) \times \left( \prod_{i=j+1}^{k} R / \B_i^{e_i} \right) \times \left(\prod_{i=k+1}^{l} R / \B_i^{e_i}  \right),
\end{equation}
where the notations are as in Section \ref{dede_sec}, but the indices $j, k$ and $l$, the ideals $\B_i$ and the integers $e_i$ are not the same as in (\ref{I_c_fac}).

Let 
\begin{equation}\label{alpha_fac}
(\alpha) = \left( \prod_{i=1}^l \B_i^{f_i} \right) \cdot J,
\end{equation}
for some non-negative integers $f_i$, with $i \in [1,l]$,  and some ideal $J$ such that no ideal $\B_i$ appears in the factorization of $J$.

We define
\begin{equation*}
d := \max_{i \in [1,l]} \left\{\left\lceil \frac{e_i}{f_i} \right\rceil: f_i \not = 0  \right\},
\end{equation*}
provided that this latter set is non-empty.

If $\tilde{x} \in \Pro(\F_{q^n})$ is $r$-periodic, then 
\begin{itemize}
\item $T^S (\tilde{x})$ denotes the tree rooted in $\tilde{x}$ and formed by elements of $S$;
\item $T_h^{S} (\tilde{x})$ denotes the set of all vertices belonging to the level $h$ of $T^{S} (\tilde{x})$ for some $h \in [0,d]$;
\item $T^{q^n} (\tilde{x})$ denotes the tree rooted in $\tilde{x}$ and formed by elements of $\Pro(\F_{q^n})$.
\end{itemize}

If $x' \in G^{q^n})$, then we denote by 
\begin{equation*}
\Pre^S (x') := \{x \in S : r(x) = x' \}
\end{equation*}  
the preimage set of $x'$ in $S$ with respect to the map $r$.

\begin{remark}
We notice that $T^S (\tilde{x}) \subseteq T^{q^n} (\tilde{x})$, namely $T^S (\tilde{x})$ is a subtree of $T^{q^n} (\tilde{x})$.

More in detail, we can distinguish the following cases.
\begin{itemize}
\item If $\tilde{x} \in S \backslash E_0$, then any predecessor of $\tilde{x}$ in $T^{q^n} (\tilde{x})$ belongs to $S \backslash E_0$ because of Lemma \ref{lem_a_n_b_n} and Lemma \ref{lem_e_0}. Therefore $T^S (\tilde{x}) = T^{q^n} (\tilde{x})$. 
\item If $\tilde{x} \in E_0$, then $T^{A_n} (\tilde{x})$ and $T^{B_n} (\tilde{x})$ can be proper subsets of $T^{q^n} (\tilde{x})$. Hence, 
\begin{equation*}
T^{q^n} (\tilde{x}) = T^{A_n} (\tilde{x}) \cup T^{B_n} (\tilde{x}),
\end{equation*} 
with $T^{A_n} (\tilde{x}) \cap T^{B_n} (\tilde{x}) \subseteq E_0$. 
\end{itemize}
\end{remark}

\begin{remark}
We notice  that $|\Pre^S (x')|$ is equal to the indegree of the vertex $x'$ in $G^{q^n}$, unless some element of $E_0$ is a preimage of $x'$ and is not contained in $S$.
\end{remark}

We introduce some notations.
\begin{itemize}
\item For any non-negative integer $s$ we define
\begin{equation*}
m_i^s := \min \{e_i, f_i  s \}.
\end{equation*}
\item For any $h \in [1,d]$ we define
\begin{align*}
v_h & := \prod_{i=1}^{l} N (\B_i^{m^h_i}) - \prod_{i=1}^{l} N (\B_i^{m^{h-1}_i}),\\
\tilde{v}_h (\tilde{x}) & := |T_h^{S} (\tilde{x}) \cap E_0|.
\end{align*}

\item If $h \in [0, d-1]$ and $i \in [1,l]$, then we define 
\begin{align*}
\nu_i^h & := N( \B_i^{\min \{f_i, m_i^h \}}),\\
\nu^h & := \prod_{i=1}^l \nu_i^h.
\end{align*}

\item If $x_h \in T^S_h (\tilde{x})$ for some $h \in [0,d-1]$, then we define
\begin{equation*}
\tilde{\nu} (x_h) := |\Pre^{S} (x_h) \cap E_0|.
\end{equation*} 

\end{itemize}

\begin{remark}
We notice that $0 \leq \tilde{v}_h (\tilde{x}) \leq 3$ and $\tilde{v}_h (\tilde{x})$ can be non-zero only if $h \leq 2$, since the rational points (different from the point at infinity) having $x$-coordinate in $E_0$ have order $2$. 
\end{remark}

Using the same notations as in Section \ref{cyc_sec} we get the following.

\begin{theorem}\label{tree_struc}
Let $f_i > 0$ for some $i \in [1,l]$.

\begin{enumerate}[leftmargin=*]
\item The tree $T^{S} (\tilde{x})$ has depth $d$ and, for any $h \in [1,d]$, we have that

\begin{displaymath}
| T^{S}_h (\tilde{x}) | =
\begin{cases}
\frac{1}{2} \cdot \left[ v_h + \tilde{v}_h (\tilde{x}) \right] & \text{if $\tilde{x} \in E_0$,} \\
v_h    & \text{if $\tilde{x} \not \in E_0$.}
\end{cases} 
\end{displaymath}

\item If $x_h \in T_h^S (\tilde{x})$ for some $h \in [0,d-1]$ and $\Pre^S (x_h) \backslash E_0 \not = \emptyset$ , then 
\begin{displaymath}
|\Pre^S (x_h) \backslash E_0| =
\begin{cases}
\frac{1}{2} (\nu^{h+1} - \tilde{\nu} (x_h)) & \text{if $x_h \in E_0$;}\\
\nu^{h+1} & \text{if $x_h \not \in E_0$.}
\end{cases}
\end{displaymath}
\end{enumerate}
\end{theorem} 
\begin{proof}
Let $(P,[0]) \in R/ I_c \times R / I_t \cong R / (\delta)$ be one of the points having $\tilde{x}$ as $x$-coordinate in $E(\F_{q^{2n}})$.

We prove separately the claims.
\begin{enumerate}[leftmargin=*]
\item We notice that any vertex of $T^{S} (\tilde{x})$ is the $x$-coordinate of a point of the form $\tilde{V} = (\tilde{P},Q) \in R / (\delta)$ such that $\alpha^u (\tilde{V}) \in \{ (\pm P, [0]) \}$ for some non-negative integer $u \leq d$. 

According to the factorization (\ref{alpha_fac}) of $(\alpha)$, we have that
\begin{equation*}
d f_i \geq \frac{e_i}{f_i} f_i = e_i
\end{equation*}
for any $i \in [1,l]$ with $f_i > 0$. Therefore the depth of $T^{S} (\tilde{x})$ is at most $d$. 
 
Consider the $x$-coordinate $x'$ of the point $V := ([\alpha]^{-d} P, [1])$. As a consequence of the definition of $d$ we have that $[\alpha]^u V \not \in \{ (\pm P, [0]) \}$ for any non-negative integer $u$ smaller than $d$, while $[\alpha]^d V \in \{(\pm P, [0]) \}$. Hence we  conclude that $T^{S} (\tilde{x})$ has depth $d$. 

Now consider the $x$-coordinate $x'$ of a point $V := (P', Q)$ such that $x'$ belongs to $T^{S}_h (\tilde{x}$) for some $h \geq 1$. Then $h$ is the smallest of the positive integers $u$ such that $\alpha^u (V) \in \{(\pm P, [0]) \}$. In particular, $P' \in \{\pm [\alpha]^{-h} P \}$.

Moreover
\begin{equation*}
Q \in (R/I_t)_h,
\end{equation*}
where
\begin{equation*}
(R/I_t)_h := \left( \prod_{i=1}^{l} \B_i^{e_i-m^h_i} / \B_i^{e_i} \right) \backslash \left( \prod_{i=1}^{l} \B_i^{e_i-m^{h-1}_i} / \B_i^{e_i} \right).
\end{equation*}
We have that
\begin{eqnarray*}
\left| \prod_{i=1}^{l} \B_i^{e_i-m^h_i} / \B_i^{e_i}  \right| & = & \prod_{i=1}^{l} N(\B_i^{m^h_i}),\\
\left| \prod_{i=1}^{l} \B_i^{e_i-m^{h-1}_i} / \B_i^{e_i} \right| & = & \prod_{i=1}^{l} N(\B_i^{m^{h-1}_i}).
\end{eqnarray*}

We deal separately with different cases.

\begin{itemize}
\item \emph{Case 1:} $(P,[0])$ has order $1$ or $2$. 

Then $\tilde{x} \in E_0$ and $P = -P$. The result follows because $([\alpha]^{-h} P, Q)$ and $([\alpha]^{-h} P, -Q)$ have the same $x$-coordinate and $Q \not = - Q$, 
unless $([\alpha]^{-h} P,Q)$ has order $2$.

\item \emph{Case 2:} $(P,[0])$ has order greater than $2$. 

The result follows because $([\alpha]^{-h} P, Q)$ and $(- [\alpha]^{-h} P, Q)$ have different $x$-coordinates, while $([\alpha]^{-h} P, Q)$ and $(- [\alpha]^{-h} P, -Q)$ have the same $x$-coordinate. 
\end{itemize}

\item Let $V_h := ([\alpha]^{-h} P,Q)$ be a point in $R/(\delta)$ whose $x$-coordinate is $x_h$.

The endomorphism $\alpha$ induces a $\nu_i^{h+1}$-$1$ map on  $\B_i^{e_i-m^{h+1}_i} / \B_i^{e_i}$ for any $i \in [1,l]$. Therefore $\alpha$ acts as a $\nu^{h+1}$-$1$ map on $\prod_{i=1}^{l} \B_i^{e_i-m^{h+1}_i} / \B_i^{e_i}$. 

Let 
\begin{equation*}
[\alpha]^{-1} (V_h) := \{([\alpha]^{-(h+1)} P, \tilde{Q}_1), \dots, ([\alpha]^{-(h+1)} P, \tilde{Q}_{\nu^{h+1}})  \}
\end{equation*} 
be the set of preimages of $([\alpha]^{-h} P,Q)$ in $R / (\delta)$.

If $Q \not = -Q$, then $|[\alpha]^{-1} (V_h)| = \nu^{h+1}$, while $|[\alpha]^{-1} (V_h)| = \frac{1}{2} \nu^{h+1}$ if $Q = -Q$.
\end{enumerate}
\end{proof}

\section{Examples}\label{exm}
In this section we construct explicitly some graphs and show how their structure relates to the results previously proved.

We recall that, if $q$ is a power of a prime and $n$ is a positive integer, then 
\begin{equation*}
\Pro (\F_{q^n}) = \{ g^m: m \in [0, n-2] \} \cup \{0, \infty \}
\end{equation*}
for some primitive element $g \in \F_{q^n}$. 

In the examples we label the vertices of $G^{q^n}$ using the exponents $m \in [0,n-2]$, the symbol $\infty$ and the symbol `0' for the zero of the field.

\begin{example}
Consider the ordinary elliptic curve 
\begin{equation*}
E : y^2 = x^3-x
\end{equation*}
defined over $\F_{73}$, whose endomorphism ring $R$ is isomorphic to $\Z[i]$.

Using Sage \cite{sage} we find  an isogeny $\sigma$ of degree $10$ from $E$ to the curve
\begin{equation*}
E' : y^2 = x^3+8x
\end{equation*} 
defined over $\F_{73}$ as 
\begin{equation*}
\sigma(x,y) := (\sigma_1(x), y \cdot \sigma_2(x)),
\end{equation*}
where
\begin{equation*}
\sigma_1(x) := \frac{x^{10}-3 x^8 + 5 x^6 -5 x^4 +3 x^2 -1}{x^9+28x^7-21x^5+28x^3+x}.
\end{equation*}

The curve $E'$ is isomorphic to the curve $E$ via the map
\begin{displaymath}
\begin{array}{rccc}
\psi : & E' & \to & E \\
& (x,y) & \mapsto & (-3x, 22y).
\end{array}
\end{displaymath}

Therefore, there exists an endomorphism $\alpha:=\psi \circ \sigma$ of degree $10$ of $E$ defined as 
\begin{equation*}
\alpha(x,y) = \left(\alpha_1(x), y \cdot \alpha_2(x) \right),
\end{equation*}
where
\begin{equation*}
\alpha_1(x) := -3 \cdot \sigma_1(x).
\end{equation*} 

We define the map $r$ as in the Introduction. 

The graph  associated with the action of $r$ over $\Pro(\F_{73})$ is represented below.

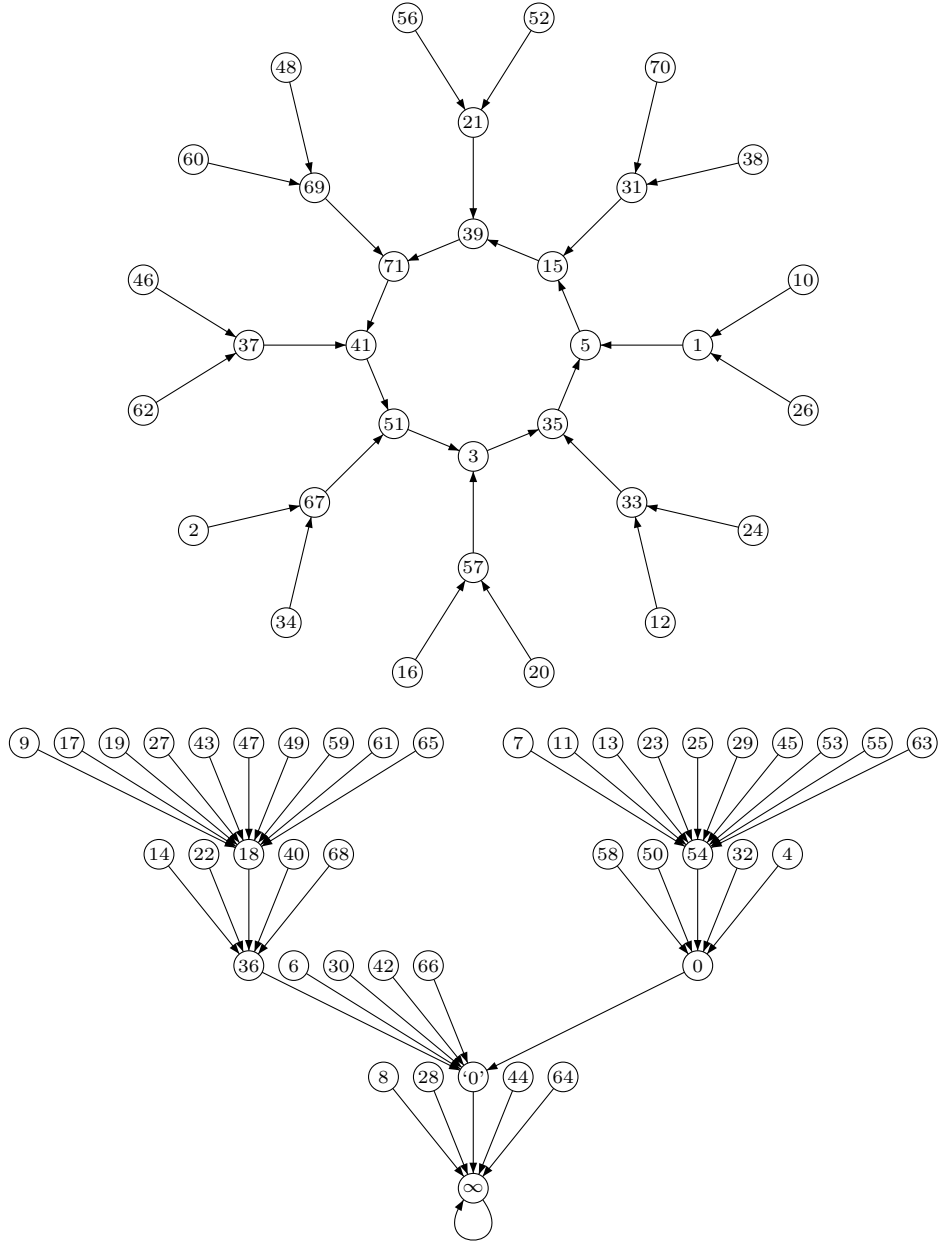
\begin{center}
\begin{picture}(90, 90)(-45,-45)
	\unitlength=2.8pt
    \gasset{Nw=4,Nh=4,Nmr=2,curvedepth=0}
    \thinlines
   	\scriptsize
    \node(A1)(15,0){5}
    \node(A2)(10.6,10.6){15}
    \node(A3)(0,15){39}
    \node(A4)(-10.6,10.6){71}
    \node(A5)(-15,0){41}
    \node(A6)(-10.6,-10.6){51}
    \node(A7)(0,-15){3}
    \node(A8)(10.6,-10.6){35}
    
    \node(B1)(30,0){1}
    \node(B2)(21.2,21.2){31}
    \node(B3)(0,30){21}
    \node(B4)(-21.2,21.2){69}
    \node(B5)(-30,0){37}
    \node(B6)(-21.2,-21.2){67}
    \node(B7)(0,-30){57}
    \node(B8)(21.2,-21.2){33}
    
    \node(C12)(44.1,8.8){10}
    \node(C21)(37.4,25){38}
    \node(C22)(25,37.4){70}
    \node(C31)(8.8,44.1){52}
    \node(C32)(-8.8,44.1){56}
    \node(C41)(-25,37.4){48}
    \node(C42)(-37.4,25){60}
    \node(C51)(-44.1,8.8){46}
    \node(C52)(-44.1,-8.8){62}
    \node(C61)(-37.4,-25){2}
    \node(C62)(-25,-37.4){34}
    \node(C71)(-8.8,-44.1){16}
    \node(C72)(8.8,-44.1){20}
    \node(C81)(25,-37.4){12}
    \node(C82)(37.4,-25){24}
    \node(C11)(44.1,-8.8){26}
 
    \drawedge(A1,A2){}
    \drawedge(A2,A3){}
    \drawedge(A3,A4){}
    \drawedge(A4,A5){}
    \drawedge(A5,A6){}
    \drawedge(A6,A7){}
    \drawedge(A7,A8){}
    \drawedge(A8,A1){}
    
    \drawedge(B1,A1){}
    \drawedge(B2,A2){}
    \drawedge(B3,A3){}
    \drawedge(B4,A4){}
    \drawedge(B5,A5){}
    \drawedge(B6,A6){}
    \drawedge(B7,A7){}
    \drawedge(B8,A8){}
    
    \drawedge(C11,B1){}
    \drawedge(C12,B1){}
    \drawedge(C21,B2){}
    \drawedge(C22,B2){}
    \drawedge(C31,B3){}
    \drawedge(C32,B3){}
    \drawedge(C41,B4){}
    \drawedge(C42,B4){}
    \drawedge(C51,B5){}
    \drawedge(C52,B5){}
    \drawedge(C61,B6){}
    \drawedge(C62,B6){}
    \drawedge(C71,B7){}
    \drawedge(C72,B7){}
    \drawedge(C81,B8){}
    \drawedge(C82,B8){}
\end{picture}
\end{center}

\begin{center}
\begin{picture}(120, 75)(-60,-10)
	\unitlength=2.8pt
    \gasset{Nw=4,Nh=4,Nmr=2,curvedepth=0}
    \thinlines
   	\scriptsize
    \node(A1)(0,0){$\infty$}

    \node(B1)(-12,15){8}
    \node(B2)(-6,15){28}
    \node(B3)(0,15){`0'}
    \node(B4)(6,15){44}
    \node(B5)(12,15){64}

    \node(C1)(-30,30){36}
    \node(C2)(-24,30){6}
    \node(C3)(-18,30){30}
    \node(C4)(-12,30){42}
    \node(C5)(-6,30){66}
    
    \node(C6)(30,30){0}

    \node(D1)(-42,45){14}
    \node(D2)(-36,45){22}
    \node(D3)(-30,45){18}
    \node(D4)(-24,45){40}
    \node(D5)(-18,45){68}
	 
    \node(D6)(42,45){4}
    \node(D7)(36,45){32}
    \node(D8)(30,45){54}
    \node(D9)(24,45){50}
    \node(D10)(18,45){58}	   
   
    \node(E11)(-60,60){9}
    \node(E21)(-54,60){17}
    \node(E31)(-48,60){19}
    \node(E41)(-42,60){27}
    \node(E51)(-36,60){43}
    \node(E61)(-30,60){47}
    \node(E71)(-24,60){49}
    \node(E81)(-18,60){59}
    \node(E91)(-12,60){61}
    \node(E101)(-6,60){65}

    \node(E12)(6,60){7}
    \node(E22)(12,60){11}
    \node(E32)(18,60){13}
    \node(E42)(24,60){23}
    \node(E52)(30,60){25}
    \node(E62)(36,60){29}
    \node(E72)(42,60){45}
    \node(E82)(48,60){53}
    \node(E92)(54,60){55}
    \node(E102)(60,60){63}
 
   	\drawedge(B1,A1){}
	\drawedge(B2,A1){}
    \drawedge(B3,A1){}
    \drawedge(B4,A1){}
    \drawedge(B5,A1){}
  
    \drawedge(C1,B3){}
    \drawedge(C2,B3){}
    \drawedge(C3,B3){}
    \drawedge(C4,B3){}
    \drawedge(C5,B3){}
    \drawedge(C6,B3){}

    \drawedge(D1,C1){}
    \drawedge(D2,C1){}
    \drawedge(D3,C1){}
    \drawedge(D4,C1){}
    \drawedge(D5,C1){}
    
    \drawedge(D6,C6){}
    \drawedge(D7,C6){}
    \drawedge(D8,C6){}
    \drawedge(D9,C6){}
    \drawedge(D10,C6){}
    
    \drawedge(E11,D3){}
    \drawedge(E21,D3){}
    \drawedge(E31,D3){}
    \drawedge(E41,D3){}
    \drawedge(E51,D3){}
    \drawedge(E61,D3){}
    \drawedge(E71,D3){}
    \drawedge(E81,D3){}
    \drawedge(E91,D3){}
    \drawedge(E101,D3){}
    
    \drawedge(E12,D8){}
    \drawedge(E22,D8){}
    \drawedge(E32,D8){}
    \drawedge(E42,D8){}
    \drawedge(E52,D8){}
    \drawedge(E62,D8){}
    \drawedge(E72,D8){}
    \drawedge(E82,D8){}
    \drawedge(E92,D8){}
    \drawedge(E102,D8){}
    
    \drawloop[loopdiam=5,loopangle=-90](A1){}
\end{picture}
\end{center}

The representations in $R$ of the Frobenius endomorphism $\pi_{73}$ and the endomorphism $\alpha$ are 
\begin{align*}
\pi_{73} & = -3 + 8 i,\\
\alpha & = 3 - i.
\end{align*}

If we set $\delta := \pi_{73}-1$, then we have that
\begin{align*}
(\delta) & =  (1+i)^4 \cdot (1-2i),\\
(\alpha) & = (1+i) \cdot (1-2i).
\end{align*}
Given that $(\delta) \subseteq (\alpha)$, the only $r$-periodic element in $A_1$ is $\infty$.

We can study the structure of the tree $T^{A_1} (\infty)$ relying upon 
\begin{equation*}
R / I_t :=  R / \B_1^4  \times R / \B_2,
\end{equation*}
where 
\begin{equation*}
\B_1 := (1+i), \quad \B_2 := (1-2i).
\end{equation*}

Using the notations of Section \ref{tree_sec} we notice that 
\begin{displaymath}
\begin{array}{ll}
e_1 = 4, & e_2 = 1,\\
f_1 = 1, & f_2 = 1,
\end{array}
\end{displaymath}
and $d=4$. 

According to Theorem \ref{tree_struc}, $T^{A_1} (\infty)$ has depth $4$. Moreover, 
\begin{eqnarray*}
|T_1^{A_1} (\infty)| & = & \frac{1}{2} \cdot [2 \cdot 5 - 1 + 1] = 5,\\
|T_2^{A_1} (\infty)| & = & \frac{1}{2} \cdot [4 \cdot 5 - 2 \cdot 5 + 2] = 6,\\
|T_3^{A_1} (\infty)| & = & \frac{1}{2} \cdot [8 \cdot 5 - 4 \cdot 5] = 10,\\
|T_4^{A_1} (\infty)| & = & \frac{1}{2} \cdot [16 \cdot 5 - 8 \cdot 5] = 20.
\end{eqnarray*}

As regards the indegrees of the vertices of $T_1^{A_1} (\infty)$, we notice that
\begin{equation*}
\nu^1 = \nu^2 = \nu^3 = 10.
\end{equation*}

Therefore 
\begin{equation*}
|\Pre^{A_1} (\infty) \backslash E_0| = 5.
\end{equation*}
Moreover, for any vertex $x_h \in T^{A_1} (\infty)$ with $h \in [1,3]$, we have that
\begin{itemize}
\item $|\Pre^{A_1} (x_h) \backslash E_0| \in \{0, 5 \}$ if $x_h \in E_0$;
\item $|\Pre^{A_1} (x_h) \backslash E_0| \in \{0, 10 \}$ if $x_h \not \in E_0$.
\end{itemize}

Now we define $\delta := \pi_{73}+1$. We have that 
\begin{equation*}
(\delta) = (1+i)^2 \cdot (-1+4i).
\end{equation*}

Therefore, the cycles formed by the elements in $B_1$ can be studied relying upon 
\begin{equation*}
R / I_c := R / \B_1,
\end{equation*}
where 
\begin{equation*}
\B_1 := (-1+4i).
\end{equation*}

According to the notations of (\ref{I_c_fac}) we have that $j=1$, $k=1$ and $l=1$, while
\begin{equation*}
H_1 = [0,1].
\end{equation*}

Let $h_1 \in H_1$. We have that
\begin{displaymath}
s_{h_1} = 
\begin{cases}
1 & \text{if $h_1 = 0$,}\\
8 & \text{if $h_1 = 1$}.
\end{cases}
\end{displaymath}

Hence, according to Theorem \ref{thm_c_h}, we have that
\begin{align*}
|C^{B_1}_{(1)}| & = \frac{1}{2 \cdot 8} \cdot 16 = 1.
\end{align*}
Moreover the only cycle in $C^{B_1}_{(1)}$ has length $8$. There is also one cycle of length $1$ formed by $\infty$ (notice in passing that the point at infinity $O$ has order $1$).

Finally we notice that the structure of the tree $T^{B_1} (\tilde{x})$ rooted in some $\tilde{x} \in B_1$ can be studied relying upon 
\begin{equation*}
R / \B_1^2,
\end{equation*}
where $\B_1 := (1+i)$.

According to Theorem \ref{tree_struc}, the tree  $T^{B_1} (\tilde{x})$ has depth $2$ and 
\begin{eqnarray*}
|T_1^{B_1} (\tilde{x})| & = & 1,\\
|T_2^{B_1} (\tilde{x})| & = & 2.
\end{eqnarray*}

In particular, we notice that $T^{B_1} (\infty)$ is a subtree of $T^{A_1} (\infty)$ and that the elements of $E_0$ are vertices of both the trees. Hence $T^{73} (\infty) = T^{A_1} (\infty)$, while $T^{B_1} (\tilde{x}) = T^{73} (\tilde{x})$ for any other $r$-periodic element $\tilde{x}$. 

Finally, since $N(\B_1) = 2$, the indegree of any vertex of the tree is $0$ or $2$ in accordance with Theorem \ref{tree_struc}. 
\end{example} 

\begin{example}\label{exm2}
Consider the ordinary elliptic curve 
\begin{equation*}
E : y^2 = x^3 + 56 x + 34
\end{equation*}
defined over $\F_{83}$, which is the reduction modulo $83$ of the curve having Cremona label 5776g1 over $\Q$ (see \cite{lmfdb}).

Such an elliptic curve has endomorphism ring $R := \Z \left[ \omega  \right]$, where $\omega = \frac{1+\sqrt{-19}}{2}$. We consider the endomorphism 
\begin{equation*}
\alpha (x,y) = (\alpha_1(x), y \cdot \alpha_2 (x))
\end{equation*}
of degree $17$, where $\alpha_1(x) = \frac{a(x)}{b(x)}$ with

\begin{eqnarray*}
a(x) & := & 4 \cdot \left( x^{17} - 32 x^{16} + 13 x^{15} + 12 x^{14} - x^{13} - 32 x^{12} + 32 x^{11} - 10 x^{10} + 8 x^9 \right. \\
& & + 2 x^8 - 8 x^7 - x^6 +x^5 + 35 x^4 + 41 x^3 + 11 x^2 +29 x +22);\\
b(x) & := & \left( x^{16} - 32 x^{15} - 39 x^{14} + 19 x^{13} + 36 x^{12} + 8 x^{11} + 41 x^{10} - 8 x^{9} - 32  x^8 \right. \\
&  & - 16 x^7 - 41 x^6 - 13 x^5 - 3 x^4 + 5 x^3 + 3 x^2 - 10x +23).
\end{eqnarray*}

The following connected components of the graph $G^{83}$ are due to the elements belonging to $A_1$. 

\begin{center}
\begin{picture}(60, 25)(-30,-10)
	\unitlength=2.8pt
    \gasset{Nw=4,Nh=4,Nmr=2,curvedepth=0}
    \thinlines
   	\scriptsize

   \node(C1)(0,0){$\infty$}
  
    \node(D2)(-21,10){17}
    \node(D3)(-15,10){32}
    \node(D4)(-9,10){37}
    \node(D5)(-3,10){40}
	 
    \node(D6)(3,10){44}
    \node(D7)(9,10){70}
    \node(D8)(15,10){74}
    \node(D9)(21,10){3}

    \drawedge(D2,C1){}
    \drawedge(D3,C1){}
    \drawedge(D4,C1){}
    \drawedge(D5,C1){}
    
    \drawedge(D6,C1){}
    \drawedge(D7,C1){}
    \drawedge(D8,C1){}
    \drawedge(D9,C1){}
    
    \drawloop[loopdiam=5,loopangle=-90](C1){}
\end{picture}
\end{center}

\begin{center}
\begin{picture}(80, 80)(-40,-40)
	\unitlength=2.8pt
    \gasset{Nw=4,Nh=4,Nmr=2,curvedepth=0}
    \thinlines
   	\scriptsize

	\node(A1)(10,17.33){53}
	\node(A2)(-20,0){23}
	\node(A3)(10,-17.33){22}  
  
	\node(D1)(40,0){16}
	\node(D2)(38.64,10.35){30}
    \node(D3)(34.65,20){33}
    \node(D4)(28.28,28.28){35}
    \node(D5)(20,34.65){43}
    \node(D6)(10.35,38.64){51}
    \node(D7)(0,40){64}
    \node(D8)(-10.35,38.64){66}
    
    	\node(E1)(-20,34.65){81}
	\node(E2)(-28.28,28.28){69}
    \node(E3)(-34.65,20){68}
    \node(E4)(-38.64,10.35){54}
    \node(E5)(-40,0){52}
    \node(E6)(-38.64,-10.35){12}
    \node(E7)(-34.65,-20){11}
    \node(E8)(-28.28,-28.28){10}
    
    	\node(F1)(-20,-34.65){19}
	\node(F2)(-10.35,-38.64){28}
    \node(F3)(0,-40){38}
    \node(F4)(10.35,-38.64){50}
    \node(F5)(20,-34.65){63}
    \node(F6)(28.28,-28.28){71}
    \node(F7)(34.65,-20){72}
    \node(F8)(38.64,-10.35){78}

   	\drawedge(D1,A1){}
   	\drawedge(D2,A1){}
   	\drawedge(D3,A1){}
   	\drawedge(D4,A1){}
    \drawedge(D5,A1){}
   	\drawedge(D6,A1){}
   	\drawedge(D7,A1){}
   	\drawedge(D8,A1){}
   	
   	\drawedge(E1,A2){}
   	\drawedge(E2,A2){}
   	\drawedge(E3,A2){}
   	\drawedge(E4,A2){}
    \drawedge(E5,A2){}
   	\drawedge(E6,A2){}
   	\drawedge(E7,A2){}
   	\drawedge(E8,A2){}
   	
   	\drawedge(F1,A3){}
   	\drawedge(F2,A3){}
   	\drawedge(F3,A3){}
   	\drawedge(F4,A3){}
    \drawedge(F5,A3){}
   	\drawedge(F6,A3){}
   	\drawedge(F7,A3){}
   	\drawedge(F8,A3){}
    
    \drawedge(A1,A2){}
    \drawedge(A2,A3){}
    \drawedge(A3,A1){}

\end{picture}
\end{center}

The structure of the $2$ connected components above can be explained as follows.

In $R$ we have that
\begin{align*}
\pi_{83} & = 7 + 2 \omega,\\
\alpha & = 3 + \omega.
\end{align*}
Let $\delta:=\pi_{83}-1$. Then
\begin{align*}
(\delta) & = (2) \cdot (3+\omega), \\
(\alpha) & = (3 + \omega). 
\end{align*}

The structure of the cycles formed by elements of $A_1$ can be determined relying upon 
\begin{equation*}
R / I_c := R / \B_1,
\end{equation*}
where $\B_1 := (2)$.

According to the notations of (\ref{I_c_fac}) we have that $j=0$, $k=1$ and $l=1$, while
\begin{equation*}
H_1 = [0,1].
\end{equation*}

Let $h_1 \in H_1$. We have that
\begin{displaymath}
s_{h_1} = 
\begin{cases}
1 & \text{if $h_1 = 0$,}\\
3 & \text{if $h_1 = 1$.}
\end{cases}
\end{displaymath}

We notice that all points in $R / \B_1$ have order $1$ or $2$: the only point of order $1$ is the point at infinity; the points of order $2$ have $x$-coordinate in $E_0$. The cycle formed by $\infty$ has length $1$, while the other elements of $E_0$ form a cycle of length $3$.

The structure of the trees can be studied relying upon
\begin{equation*}
R / I_t := R / \B_1,
\end{equation*}
where $\B_1 := (3+\omega)$.

Therefore $T^{83} (\tilde{x})$ has depth $1$ for any $\tilde{x} \in A_1$. Moreover, according to Theorem \ref{tree_struc}, the number of vertices belonging to $T_1^{83} (\tilde{x})$ is 
\begin{equation*}
| T_1^{83} (\tilde{x})| = \frac{1}{2} \cdot [17 - 1] = 8. 
\end{equation*}  

Now consider the remaining connected components of $G^{83}$.
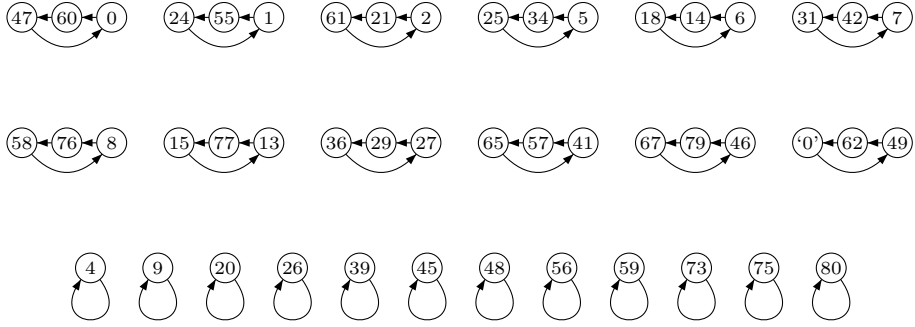
\begin{center}
\begin{picture}(120, 25)(-57,-5)
	\unitlength=2.8pt
    \gasset{Nw=4,Nh=4,Nmr=2,curvedepth=0}
    \thinlines
   	\scriptsize

	\node(A1)(-57,10){47}
	\node(A2)(-51,10){60}
	\node(A3)(-45,10){0}

	\node(B1)(-36,10){24}
	\node(B2)(-30,10){55}
	\node(B3)(-24,10){1}
	
	\node(C1)(-15,10){61}
	\node(C2)(-9,10){21}
	\node(C3)(-3,10){2} 
	
	\node(D1)(6,10){25}
	\node(D2)(12,10){34}
	\node(D3)(18,10){5}
	
	\node(E1)(27,10){18}
	\node(E2)(33,10){14}
	\node(E3)(39,10){6}
	
	\node(F1)(48,10){31}
	\node(F2)(54,10){42}
	\node(F3)(60,10){7}

   	\drawedge(A3,A2){}
   	\drawedge(A2,A1){}
   	
   	\drawedge(B3,B2){}
   	\drawedge(B2,B1){}
   	
   	\drawedge(C3,C2){}
   	\drawedge(C2,C1){}
   	
   	\drawedge(D3,D2){}
   	\drawedge(D2,D1){}
   	
   	\drawedge(E3,E2){}
   	\drawedge(E2,E1){}
   	
   	\drawedge(F3,F2){}
   	\drawedge(F2,F1){}

	\gasset{curvedepth=-4}   
	
    \drawedge(A1,A3){}
    \drawedge(B1,B3){}
    \drawedge(C1,C3){}
    \drawedge(D1,D3){}
    \drawedge(E1,E3){}
    \drawedge(F1,F3){}

\end{picture}
\end{center}

\begin{center}
\begin{picture}(120, 15)(-57,-5)
	\unitlength=2.8pt
    \gasset{Nw=4,Nh=4,Nmr=2,curvedepth=0}
    \thinlines
   	\scriptsize

	\node(A1)(-57,10){58}
	\node(A2)(-51,10){76}
	\node(A3)(-45,10){8}

	\node(B1)(-36,10){15}
	\node(B2)(-30,10){77}
	\node(B3)(-24,10){13}
	
	\node(C1)(-15,10){36}
	\node(C2)(-9,10){29}
	\node(C3)(-3,10){27} 
	
	\node(D1)(6,10){65}
	\node(D2)(12,10){57}
	\node(D3)(18,10){41}
	
	\node(E1)(27,10){67}
	\node(E2)(33,10){79}
	\node(E3)(39,10){46}
	
	\node(F1)(48,10){`0'}
	\node(F2)(54,10){62}
	\node(F3)(60,10){49}

   	\drawedge(A3,A2){}
   	\drawedge(A2,A1){}
   	
   	\drawedge(B3,B2){}
   	\drawedge(B2,B1){}
   	
   	\drawedge(C3,C2){}
   	\drawedge(C2,C1){}

   	\drawedge(D3,D2){}
   	\drawedge(D2,D1){}
   	
   	\drawedge(E3,E2){}
   	\drawedge(E2,E1){}
   	
   	\drawedge(F3,F2){}
   	\drawedge(F2,F1){}

	\gasset{curvedepth=-4}   
	
    \drawedge(A1,A3){}
    \drawedge(B1,B3){}
    \drawedge(C1,C3){}
    \drawedge(D1,D3){}
    \drawedge(E1,E3){}
    \drawedge(F1,F3){}

\end{picture}
\end{center}

\begin{center}
\begin{picture}(120, 20)(-66,-10)
	\unitlength=2.8pt
    \gasset{Nw=4,Nh=4,Nmr=2,curvedepth=0}
    \thinlines
   	\scriptsize

	\node(A1)(-57,10){4}
	\node(A2)(-48,10){9}
	\node(A3)(-39,10){20}

	\node(B1)(-30,10){26}
	\node(B2)(-21,10){39}
	\node(B3)(-12,10){45}
	
	\node(C1)(-3,10){48}
	\node(C2)(6,10){56}
	\node(C3)(15,10){59} 
	
	\node(D1)(24,10){73}
	\node(D2)(33,10){75}
	\node(D3)(42,10){80}

    \drawloop[loopdiam=5, loopangle=-90](A1){}
    \drawloop[loopdiam=5, loopangle=-90](A2){}
    \drawloop[loopdiam=5, loopangle=-90](A3){}

    \drawloop[loopdiam=5, loopangle=-90](B1){}
    \drawloop[loopdiam=5, loopangle=-90](B2){}
    \drawloop[loopdiam=5, loopangle=-90](B3){}
    
    \drawloop[loopdiam=5, loopangle=-90](C1){}
    \drawloop[loopdiam=5, loopangle=-90](C2){}
    \drawloop[loopdiam=5, loopangle=-90](C3){}

    \drawloop[loopdiam=5, loopangle=-90](D1){}
    \drawloop[loopdiam=5, loopangle=-90](D2){}
    \drawloop[loopdiam=5, loopangle=-90](D3){}

\end{picture}
\end{center}

All the vertices in the connected components of length $1$ and $3$ above are due to elements belonging to $B_1$, which are the $x$-coordinates of the rational points contained in 
\begin{equation*}
R / (\delta),
\end{equation*}
where $\delta := \pi_{83} + 1$. 

We have that 
\begin{equation*}
(\delta) =  (1 - \omega)^2 \cdot (2).
\end{equation*}

Therefore the elements of $B_1$ are $r$-periodic. 

Let
\begin{equation*}
R / I_c :=   R / \B_1^2 \times R / \B_2,
\end{equation*}
where
\begin{displaymath}
\begin{array}{ll}
\B_1  := (1 - \omega), & \B_2  := (2).
\end{array}
\end{displaymath}  

According to the notations of (\ref{I_c_fac}) we have that $j=1$, $k=2$ and $l=2$, while
\begin{equation*}
H_1 = [0,2], \quad H_2 = [0,1].
\end{equation*}

Let $h_1 \in H_1$. We have that
\begin{displaymath}
s_{h_1} = 
\begin{cases}
1 & \text{if $h_1 = 0$,}\\
1 & \text{if $h_1 = 1$,}\\
1 & \text{if $h_1 = 2$.}\\
\end{cases}
\end{displaymath}
Moreover $\alpha^{s_{h_1}} + 1 \in \B_1^{h_1}$ for $h_1 \in \{1, 2 \}$.

Let $h_2 \in H_2$. We have that
\begin{displaymath}
s_{h_2} = 
\begin{cases}
1 & \text{if $h_2 = 0$,}\\
3 & \text{if $h_2 = 1$.}
\end{cases}
\end{displaymath} 
Moreover $\alpha^{s_{h_2}} + 1 \in \B_2$ for $h_2 = 1$.

Hence, according to Theorem \ref{thm_c_h}, we have that
\begin{align*}
|C^{B_1}_{(1,0)}| & = \frac{1}{2 \cdot 1} \cdot 4 = 2,\\
|C^{B_1}_{(2,0)}| & = \frac{1}{2 \cdot 1} \cdot 20 = 10,
\end{align*}
and any cycle in $C^{B_1}_{(1,0)}$ and $C^{B_1}_{(2,0)}$ has length $1$.

Moreover we have that
\begin{align*}
|C^{B_1}_{(1,1)}| & = \frac{1}{2 \cdot 3} \cdot (4 \cdot 3) = 2,\\
|C^{B_1}_{(2,1)}| & = \frac{1}{2 \cdot 3} \cdot (20 \cdot 3) = 10,
\end{align*}
and any cycle in $C^{B_1}_{(1,1)}$ and $C^{B_1}_{(2,1)}$ has length $3$.

In $R / I_c$ we can find one point of order $1$ and $3$ points of order two. Indeed, $\infty$ forms a cycle of length $1$, while the $x$-coordinates of the points of order $2$ form a cycle of length $3$, namely the cycle whose vertices are labelled by $22$, $23$ and $53$. We notice that such a cycle is formed by elements in $A_1$ too.

In conclusion, the elements of $B_1$ form $13$ cycles of length $1$ and $13$ cycles of length $3$. Two of these cycles are also formed by elements of $A_1$.
\end{example}

\begin{example}\label{exm3}
Consider the ordinary elliptic curve 
\begin{equation*}
E : y^2 = x^3 + x + \gamma
\end{equation*}
defined over $\F_{25}$, where $\gamma$ is a primitive element of $\F_{25}$.

The endomorphism ring of the curve $E$ is $R := \Z \left[ \omega  \right]$, where $\omega = \sqrt{-21}$. 

We consider the endomorphism 
\begin{equation*}
\alpha (x,y) = (\alpha_1(x), y \cdot \alpha_2 (x))
\end{equation*}
of degree $22$, where $\alpha_1(x) = \frac{a(x)}{b(x)}$ with

\begin{eqnarray*}
a(x) & := & x^{22} -\gamma  x^{21} -2 \gamma  x^{20} + (-\gamma + 1) x^{19} + (\gamma + 2) x^{18} + x^{17} -2 (\gamma +1) x^{16} \\
& &  -2 (\gamma -1) x^{15} + x^{14} + \gamma x^{13} - (2 \gamma + 1) x^{12} + 2 \gamma x^{11} + (\gamma - 1) x^{10} + 2 \gamma x^9 \\
& &  - (\gamma + 2) x^8 - (\gamma - 2) x^7 - 2 x^6 + (2\gamma + 1) x^5 - (\gamma - 1) x^4 + 2 \gamma x^3 + 2 \gamma x^2 \\
& & + (2 \gamma + 1) x + 2 \gamma;\\
b(x) & := & x^{21} -\gamma x^{20} -2 \gamma x^{19} -(\gamma - 1) x^{18} + (\gamma + 2) x^{17} + x^{16} + -2 (\gamma +1) x^{15} \\
& & -2 (\gamma -1) x^{14} + x^{13} + \gamma x^{12} -(2 \gamma + 1) x^{11} + 2 \gamma x^{10} + (\gamma + 2) x^9 -\gamma x^8 \\
& & -2 (\gamma -1) x^7 + \gamma x^6 + (\gamma + 1) x^5 -(\gamma + 1) x^4 + 2 x^3 - (2\gamma - 1) x^2 -\gamma x \\
& & -2 \gamma - 1.
\end{eqnarray*}

In this example we construct the graph associated with $r$ over $\F_{25^2}$. 

According to the notations of the Introduction, we have 
\begin{equation*}
q = 25, \quad n =2.
\end{equation*}

Using Sage we compute the minimal polynomial $x^2 - 4 x + 25$ of the Frobenius endomorphism $\pi_{25}$, whose representation as an element of $R$ is 
\begin{equation*}
\pi_{25} = 2 + \omega.
\end{equation*}
The representation of $\alpha$ as an endomorphism of $R$ is 
\begin{equation*}
\alpha = 1 + \omega.
\end{equation*}

Since $|\Pro (\F_{25^2}) | = 626$, we introduce some symbols with the aim to have a more compact representation of $G^{25^2}$.
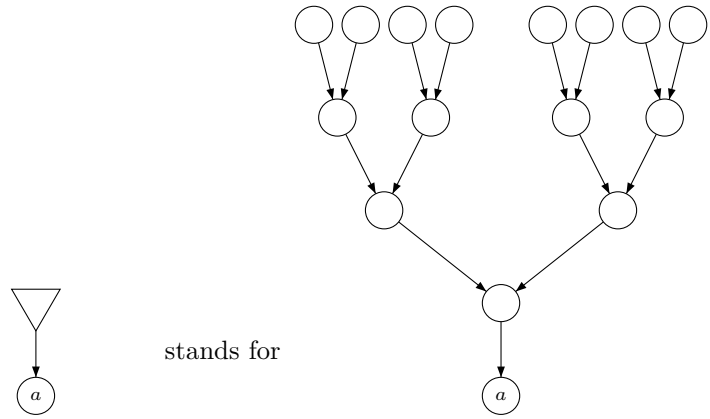
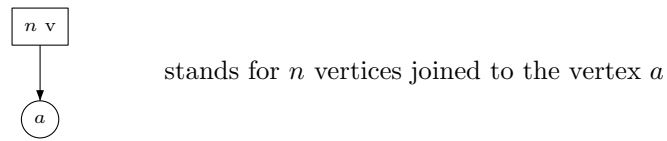
\begin{center}
\begin{picture}(20, 20)(10,-10)
	\unitlength=3.5pt
    \gasset{Nw=4,Nh=4,Nmr=2,curvedepth=0}
    \thinlines
   	\scriptsize

    \node(A1)(0,0){$a$}

    \rpnode[polyangle=150,iangle=45](T1)(0,10)(3,3){}
    \drawedge(T1,A1){} 
    \normalsize
    \node[Nframe=n](stand)(20,5){stands for}
\end{picture}
\begin{picture}(40, 60)(-30,-10)
	\unitlength=3.5pt
    \gasset{Nw=4,Nh=4,Nmr=2,curvedepth=0}
    \thinlines
   	\scriptsize

    \node(A1)(0,0){$a$}
    
    \node(A11)(0,10){}
    
    \node(A111)(-12.5,20){}
    \node(A112)(12.5,20){}    
    
    \node(A1111)(-17.5,30){}
    \node(A1112)(-7.5,30){}
    \node(A1121)(7.5,30){}
    \node(A1122)(17.5,30){}
    
    \node(A11111)(-20,40){}
    \node(A11112)(-15,40){}
    \node(A11121)(-10,40){}
    \node(A11122)(-5,40){}
    \node(A11211)(5,40){}
    \node(A11212)(10,40){}
    \node(A11221)(15,40){}
    \node(A11222)(20,40){}

    \drawedge(A11,A1){} 
    
    \drawedge(A111,A11){} 
    \drawedge(A112,A11){}
    
    \drawedge(A1111,A111){} 
    \drawedge(A1112,A111){}
    \drawedge(A1121,A112){} 
    \drawedge(A1122,A112){}
        
    \drawedge(A11111,A1111){} 
    \drawedge(A11112,A1111){}
    \drawedge(A11121,A1112){} 
    \drawedge(A11122,A1112){}   
    \drawedge(A11211,A1121){} 
    \drawedge(A11212,A1121){}
    \drawedge(A11221,A1122){} 
    \drawedge(A11222,A1122){}
\end{picture}
\end{center}

\begin{center}
\begin{picture}(20, 30)(30,-10)
	\unitlength=3.5pt
    \gasset{Nw=4,Nh=4,Nmr=2,curvedepth=0}
    \thinlines
   	\scriptsize

    \node(A1)(0,0){$a$}

    \gasset{Nw=6,Nh=4,Nmr=0,curvedepth=0}    
    \node(A2)(0,10){$n$ v}
    \drawedge(A2,A1){} 
    \normalsize
    \node[Nframe=n](stand)(40,5){stands for $n$ vertices joined to the vertex $a$}
\end{picture}
\end{center}

Here are the five connected components of $G^{25^2}$.

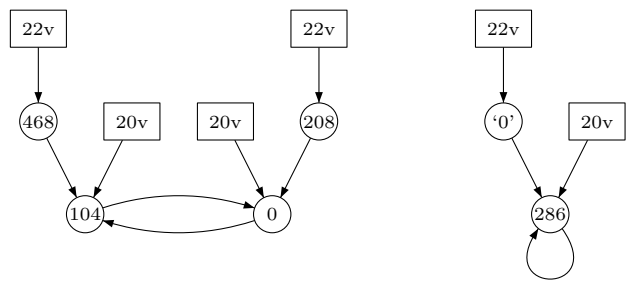
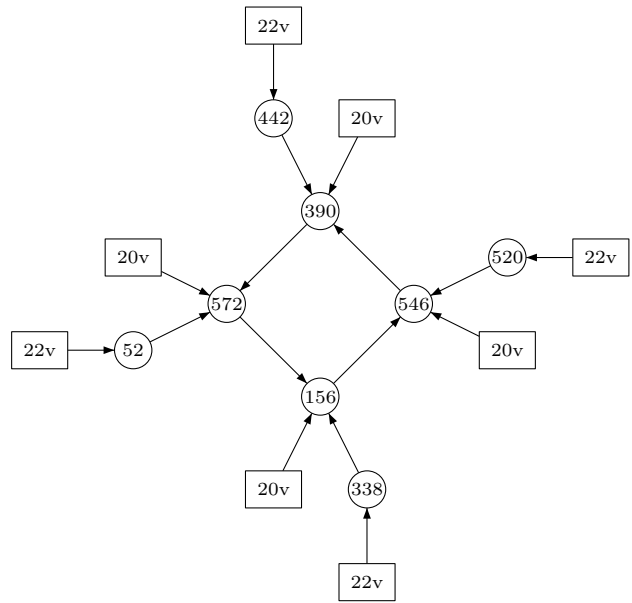
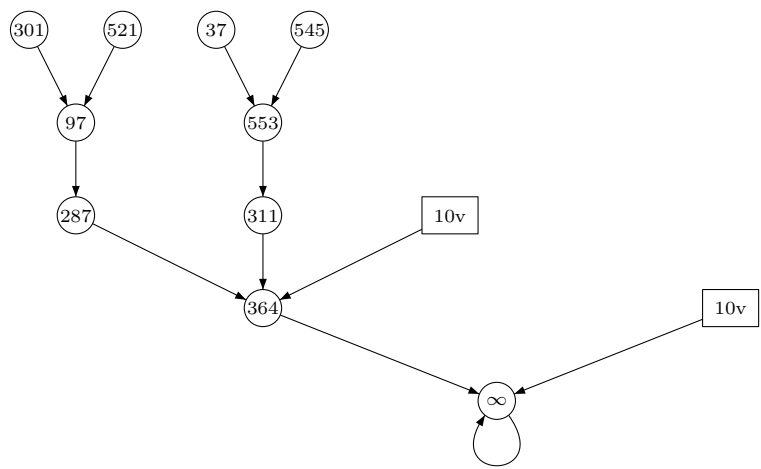
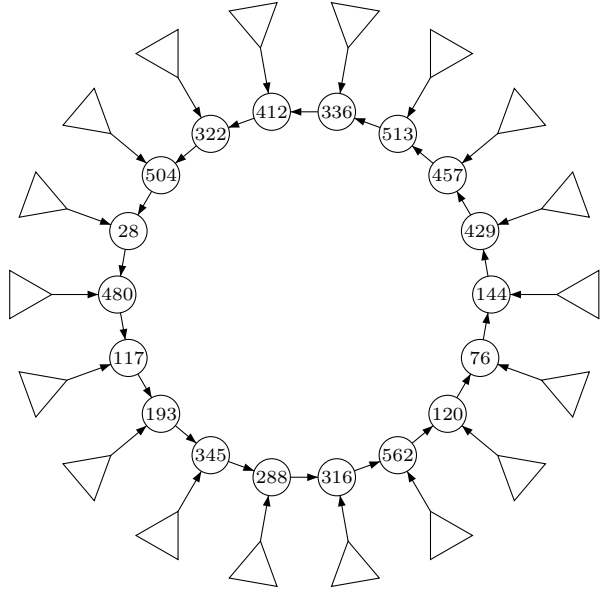
\begin{center}
\begin{picture}(40, 30)(10,-10)
	\unitlength=3.5pt
    \gasset{Nw=4,Nh=4,Nmr=2,curvedepth=0}
    \thinlines
   	\scriptsize

    \node(A1)(10,0){$104$}
    \node(A2)(30,0){$0$}
    
    \node(B1)(5,10){468}

    \node(B4)(35,10){208}

	\gasset{Nw=6,Nh=4,Nmr=0,curvedepth=0}
	\node(B2)(15,10){20v}
	\node(B3)(25,10){20v}
	
	\node(C1)(5,20){22v}
    \node(C2)(35,20){22v}

    \drawedge(C1,B1){}
    \drawedge(B1,A1){}
    \drawedge(C2,B4){}
    \drawedge(B4,A2){}
   
   \drawedge(B2,A1){}
   \drawedge(B3,A2){}
   
   \gasset{curvedepth=2}
   \drawedge(A1,A2){}
   \drawedge(A2,A1){}

\end{picture}
\begin{picture}(40, 30)(-10,-10)
	\unitlength=3.5pt
    \gasset{Nw=4,Nh=4,Nmr=2,curvedepth=0}
    \thinlines
   	\scriptsize

    \node(A1)(10,0){286}
       
    \node(B1)(5,10){`0'}

	\gasset{Nw=6,Nh=4,Nmr=0,curvedepth=0}
	\node(B2)(15,10){20v}
		
	\node(C1)(5,20){22v}

   \drawedge(B2,A1){}
   \drawedge(B1,A1){}
    \drawedge(C1,B1){}

    \drawloop[loopdiam=5,loopangle=-90](A1){}
\end{picture}
\end{center}

\begin{center}
\begin{picture}(100, 100)(-50,-50)
	\unitlength=3.5pt
    \gasset{Nw=4,Nh=4,Nmr=2,curvedepth=0}
    \thinlines
   	\scriptsize

    \node(A1)(0,10){390}
    \node(A2)(-10,0){572}
    \node(A3)(0,-10){156}
    \node(A4)(10,0){546}
    
    \node(B11)(-5,20){442}
    
    \node(B21)(-20,-5){52}
    
    \node(B31)(5,-20){338}
    
    \node(B41)(20,5){520}

	\gasset{Nw=6,Nh=4,Nmr=0,curvedepth=0}
	\node(B12)(5,20){20v}
	\node(C11)(-5,30){22v}
	
	\node(B22)(-20,5){20v}
	\node(C21)(-30,-5){22v}
	
	\node(B32)(-5,-20){20v}
	\node(C31)(5,-30){22v}
	
	\node(B42)(20,-5){20v}
	\node(C41)(30,5){22v}

    \drawedge(A1,A2){}
    \drawedge(A2,A3){}
    \drawedge(A3,A4){}
    \drawedge(A4,A1){}
  
   \drawedge(B12,A1){}
   \drawedge(B11,A1){}
   \drawedge(C11,B11){}
   
   \drawedge(B22,A2){}
   \drawedge(B21,A2){}
   \drawedge(C21,B21){}
   
   \drawedge(B32,A3){}
   \drawedge(B31,A3){}
   \drawedge(C31,B31){}
   
   \drawedge(B42,A4){}
   \drawedge(B41,A4){}
   \drawedge(C41,B41){}
\end{picture}
\end{center}
\begin{center}
\begin{picture}(100, 70)(0,-10)
	\unitlength=3.5pt
    \gasset{Nw=4,Nh=4,Nmr=2,curvedepth=0}
    \thinlines
   	\scriptsize

    \node(A1)(50,5){$\infty$}
  
    \node(B11)(25,15){364}

    \node(C111)(5,25){287}
    \node(C112)(25,25){311}

    \node(D1111)(5,35){97}
    \node(D1121)(25,35){553}

    \node(E11111)(0,45){301}
    \node(E11112)(10,45){521}
    \node(E11211)(20,45){37}
    \node(E11212)(30,45){545}

	\gasset{Nw=6,Nh=4,Nmr=0,curvedepth=0}
    \node(B12)(75,15){10v}
    \node(C113)(45,25){10v}

    \drawedge(B11,A1){}
    \drawedge(B12,A1){}
    
    \drawedge(C111,B11){}
    \drawedge(C112,B11){}
    \drawedge(C113,B11){}
    
    \drawedge(D1111,C111){}
    \drawedge(D1121,C112){}
    
    \drawedge(E11111,D1111){}
    \drawedge(E11112,D1111){}
    \drawedge(E11211,D1121){}
    \drawedge(E11212,D1121){}
    
    \drawloop[loopdiam=5,loopangle=-90](A1){}
\end{picture}
\end{center}

\begin{center}
\begin{picture}(100, 80)(-50,-40)
	\unitlength=3.5pt
    \gasset{Nw=4,Nh=4,Nmr=2,curvedepth=0}
    \thinlines
   	\scriptsize

    \node(A1)(20,0){$144$}
    \node(A2)(18.8,6.84){429}
    \node(A3)(15.32,12.86){457}
    \node(A4)(10,17.32){513}
    \node(A5)(3.48,19.7){336}
    \node(A9)(-18.8,6.84){28}
    \node(A8)(-15.32,12.86){504}
    \node(A7)(-10,17.32){322}
    \node(A6)(-3.48,19.7){412}
    \node(A10)(-20,0){480}
    \node(A11)(-18.8,-6.84){117}
    \node(A12)(-15.32,-12.86){193}
    \node(A13)(-10,-17.32){345}
    \node(A14)(-3.48,-19.7){288}
    \node(A18)(18.8,-6.84){76}
    \node(A17)(15.32,-12.86){120}
    \node(A16)(10,-17.32){562}
    \node(A15)(3.48,-19.7){316}
    
    \rpnode[polyangle=60,iangle=45](T1)(30,0)(3,3){}
    \rpnode[polyangle=80,iangle=45](T2)(28.19,10.26)(3,3){}
    \rpnode[polyangle=100,iangle=45](T3)(22.98,19.28)(3,3){}
    \rpnode[polyangle=120,iangle=45](T4)(15,25.98)(3,3){}
    \rpnode[polyangle=140,iangle=45](T5)(5.21,29.54)(3,3){}
    \rpnode[polyangle=220,iangle=45](T9)(-28.19,10.26)(3,3){}
    \rpnode[polyangle=200,iangle=45](T8)(-22.98,19.28)(3,3){}
    \rpnode[polyangle=180,iangle=45](T7)(-15,25.98)(3,3){}
    \rpnode[polyangle=160,iangle=45](T6)(-5.21,29.54)(3,3){}
    \rpnode[polyangle=240,iangle=45](T10)(-30,0)(3,3){}
    \rpnode[polyangle=260,iangle=45](T11)(-28.19,-10.26)(3,3){}
    \rpnode[polyangle=280,iangle=45](T12)(-22.98,-19.28)(3,3){}
    \rpnode[polyangle=300,iangle=45](T13)(-15,-25.98)(3,3){}
    \rpnode[polyangle=320,iangle=45](T14)(-5.21,-29.54)(3,3){}
    \rpnode[polyangle=400,iangle=45](T18)(28.19,-10.26)(3,3){}
    \rpnode[polyangle=380,iangle=45](T17)(22.98,-19.28)(3,3){}
    \rpnode[polyangle=360,iangle=45](T16)(15,-25.98)(3,3){}
    \rpnode[polyangle=340,iangle=45](T15)(5.21,-29.54)(3,3){}

    \drawedge(A1,A2){}    
    \drawedge(A2,A3){}
    \drawedge(A3,A4){}
    \drawedge(A4,A5){}
    \drawedge(A5,A6){}    
    \drawedge(A6,A7){}
    \drawedge(A7,A8){}
    \drawedge(A8,A9){}
    \drawedge(A9,A10){}    
    \drawedge(A10,A11){}
    \drawedge(A11,A12){}
    \drawedge(A12,A13){}
    \drawedge(A13,A14){}    
    \drawedge(A14,A15){}
    \drawedge(A15,A16){}
    \drawedge(A16,A17){}
    \drawedge(A17,A18){}
    \drawedge(A18,A1){}

    \drawedge(T1,A1){}
    \drawedge(T2,A2){}    
    \drawedge(T3,A3){}
    \drawedge(T4,A4){}
    \drawedge(T5,A5){}
    \drawedge(T6,A6){}    
    \drawedge(T7,A7){}
    \drawedge(T8,A8){} 
    \drawedge(T9,A9){}
    \drawedge(T10,A10){}    
    \drawedge(T11,A11){}
    \drawedge(T12,A12){}
    \drawedge(T13,A13){}
    \drawedge(T14,A14){}    
    \drawedge(T15,A15){}
    \drawedge(T16,A16){}
    \drawedge(T17,A17){}
    \drawedge(T18,A18){} 
\end{picture}
\end{center}

Now we analyse the structure of the connected components of $G^{25^2}$ formed by the elements of $A_2$.

If $\delta := \pi_{25}^2-1$, then 
\begin{align*}
(\delta) & = I_1 \cdot I_2 \cdot I_3^2 \cdot I_4,\\
(\alpha) & = I_2 \cdot I_3, 
\end{align*}
where 
\begin{align*}
I_1 & := (5,\omega+3),\\
I_2 & := (11, \omega+1),\\
I_3 & := (2, \omega+1),\\
I_4 & := (3, \omega),
\end{align*}
with $N(I_1) = 5$, $N(I_2) = 11$, $N(I_3) = 2$ and $N(I_4) = 3$.

The structure of the cycles formed by elements of $A_2$ can be determined relying upon 
\begin{equation*}
R / I_c := R / \B_1 \times R / \B_2,
\end{equation*}
where 
\begin{align*}
\B_1 & := I_1,\\
\B_2 & := I_4. 
\end{align*}

According to the notations of (\ref{I_c_fac}) we have that $j=1$, $k=1$ and $l=2$, while
\begin{equation*}
H_1 = [0,1], \quad H_2 = [0,1].
\end{equation*}

Let $h_1 \in H_1$. We have that
\begin{displaymath}
s_{h_1} = 
\begin{cases}
1 & \text{if $h_1 = 0$,}\\
2 & \text{if $h_1 = 1$.}
\end{cases}
\end{displaymath}
Moreover $\alpha + 1 \in \B_1$.

Let $h_2 \in H_2$. We have that
\begin{displaymath}
s_{h_2} = 
\begin{cases}
1 & \text{if $h_2 = 0$,}\\
1 & \text{if $h_2 = 1$.}
\end{cases}
\end{displaymath} 
Moreover $\alpha - 1 \in \B_2$.

Hence, according to Theorem \ref{thm_c_h}, we have that
\begin{align*}
|C^{A_2}_{(0,1)}| & = \frac{1}{2 \cdot 1} \cdot (1 \cdot 2) = 1,\\
|C^{A_2}_{(1,0)}| & = \frac{1}{2 \cdot 2} \cdot (4 \cdot 1) = 1,\\
|C^{A_2}_{(1,1)}| & = \frac{1}{2 \cdot 4} \cdot (4 \cdot 2) = 1,
\end{align*}
where the unique cycle in $C^{A_2}_{(0,1)}$ (resp. $C^{A_2}_{(1,0)}$, $C^{A_2}_{(1,1)}$) has length $1$ (resp. 2, 4). Such cycles can be found in the first three connected components of the graph. Finally, we notice that the point $([0], [0]) \in R / I_c$ has order $1$. Indeed, $\infty$ forms a cycle of length $1$.

Before proceeding we notice that 
\begin{equation*}
E_0 = \{\infty, g^{287}, g^{311}, g^{364} \},
\end{equation*}
where $g$ is a generator of $\F_{625}$.

Now we study the structure of the trees rooted in elements of $A_2$ relying upon 
\begin{equation*}
R / I_t := R / \B_1 \times R / \B_2^2,
\end{equation*}
where 
\begin{align*}
\B_1 & := I_2,\\
\B_2 & := I_3. 
\end{align*}

If $\tilde{x} \in A_2 \backslash E_0$, then 
\begin{align*}
|T_1^{A_2} (\tilde{x})| & = 11 \cdot 2 - 1 = 21,\\
|T_2^{A_2} (\tilde{x})| & = 11 \cdot 4 - 11 \cdot 2 = 22,
\end{align*}
while
\begin{align*}
|T_1^{A_2} (\infty)| & = \frac{1}{2} \cdot [11 \cdot 2 - 1 + 1 ] =11,\\
|T_2^{A_2} (\infty)| & = \frac{1}{2} \cdot [11 \cdot 4 - 22 + 2] = 12.
\end{align*}

We also notice that the indegree of any vertex belonging to a tree rooted in some elements of the first three connected components is $22$ since 
\begin{equation*}
\nu^1 = \nu^2 = 22
\end{equation*}
and no vertex in such components belongs to $E_0$.

Now we study the structure of the connected components formed by the elements of $B_2$.

If $\delta := \pi_{25}^2+1$, then 
\begin{align*}
(\delta) & = I_5 \cdot I_3^4,
\end{align*}
where $I_3$ is as before, while
\begin{align*}
I_5 & := (\omega-4),
\end{align*}
with $N(I_5) = 37$.

The structure of the cycles formed by the elements of $B_2$ can be determined relying upon 
\begin{equation*}
R / I_c := R / \B_1,
\end{equation*}
where $\B_1  := I_5$.

Following the notations of (\ref{I_c_fac}) we have that $j=1$, $k=1$ and $l=1$, while
\begin{equation*}
H_1 = [0,1].
\end{equation*}

Let $h_1 \in H_1$. We have that
\begin{displaymath}
s_{h_1} = 
\begin{cases}
1 & \text{if $h_1 = 0$,}\\
18 & \text{if $h_1 = 1$.}
\end{cases}
\end{displaymath}
Moreover $\alpha^{18} + 1 \in \B_1$.

Hence, according to Theorem \ref{thm_c_h}, we have that
\begin{align*}
|C^{B_2}_{(1)}| & = \frac{1}{2 \cdot 18} \cdot (36) = 1,
\end{align*}
where the unique cycle in $C^{B_2}_{(1)}$ has length $18$. Such a cycle can be found in the last connected component of the graph. As above, the point $([0]) \in R / I_c$ has order $1$ and $\infty$ forms a cycle of length $1$.

Now we study the structure of the trees rooted in elements of $B_2$ relying upon 
\begin{equation*}
R / I_t := R / \B_1^4,
\end{equation*}
where $\B_1 := I_3$. 

Using the notations of Section \ref{tree_sec} we have that 
\begin{displaymath}
\begin{array}{ll}
e_1 = 4,\\
f_1 = 1,
\end{array}
\end{displaymath}
and $d=4$. 

If $\tilde{x} \in B_2 \backslash E_0$, then 
\begin{align*}
|T_1^{B_2} (\tilde{x})| & = 2 - 1 = 1,\\
|T_2^{B_2} (\tilde{x})| & = 4 - 2 = 2,\\
|T_3^{B_2} (\tilde{x})| & = 8 - 4 = 4,\\
|T_4^{B_2} (\tilde{x})| & = 16 - 2 = 8,\\
\end{align*}
while
\begin{align*}
|T_1^{B_2} (\infty)| & = \frac{1}{2} \cdot [2 - 1 + 1 ] = 1,\\
|T_2^{B_2} (\infty)| & = \frac{1}{2} \cdot [4 - 2 + 2] = 2,\\
|T_3^{B_2} (\infty)| & = \frac{1}{2} \cdot [8 - 4 ] = 2,\\
|T_4^{B_2} (\infty)| & = \frac{1}{2} \cdot [16 - 8] = 4.
\end{align*}
\end{example}

\bibliography{Refs}
\end{document}